\documentclass[12pt,reqno]{amsart}

\usepackage[utf8]{inputenc}
\usepackage[T1]{fontenc}
\usepackage{lmodern,amssymb,bbm}

\usepackage[left=3cm,right=3cm,top=3cm,bottom=3cm]{geometry}
\usepackage[cmtip]{xy}
\xyoption{pdf}
\xyoption{color}
\xyoption{all}

\usepackage[svgnames,hyperref]{xcolor}
\newcommand{\mycolor}{Navy}
\usepackage[pdfauthor={Dat T. T\^O}, colorlinks, linktocpage, citecolor = \mycolor, linkcolor = \mycolor, urlcolor = \mycolor]{hyperref}

\theoremstyle{plain}    
 \newtheorem{theorem}{Theorem}[section]
 \numberwithin{equation}{section} 
 \numberwithin{figure}{section} 
 \theoremstyle{plain}
 \theoremstyle{plain}    
 \theoremstyle{plain}    
 \newtheorem{proposition}[theorem]{Proposition} 
 \theoremstyle{plain}    
 \newtheorem{lemma}[theorem]{Lemma} 
 \theoremstyle{plain}
 \newtheorem{remark}[theorem]{Remark}
 \theoremstyle{plain}
\newtheorem{example}[theorem]{Example}
 \theoremstyle{plain}

 \newtheorem*{thmA}{Theorem A} 
 \newtheorem*{thmB}{Theorem B}

 \newtheorem*{ackn}{Acknowledgements}

\theoremstyle{definition}
\newtheorem{definition}[theorem]{Definition}

\newcommand{\K}{\mathcal{K}}

\usepackage[all]{hypcap} 

\newcommand{\na}{\nabla}

\newcommand{\N}{{\mathbb{N}}}

\newcommand{\R}{{\mathbb{R}}}

\newcommand{\M}{{\rm M}}

\renewcommand{\a}{\alpha}

\newcommand{\e}{\varepsilon}



\begin{document}

\title{Monge-Amp\`ere equations on compact Hessian manifolds}
\author{V. Guedj and  T. D. T\^o}


\address{Institut de Math\'ematiques de Toulouse   \\ Universit\'e de Toulouse \\
118 route de Narbonne \\
31400 Toulouse, France\\}

\email{\href{mailto:vincent.guedj@math.univ-toulouse.fr}{vincent.guedj@math.univ-toulouse.fr}}
\urladdr{\href{https://www.math.univ-toulouse.fr/~guedj}{https://www.math.univ-toulouse.fr/~guedj/}}

\address{Institut de Math\'ematiques de Jussieu-Paris Rive Gauche \\ Sorbonne Universit\'e \\ 4 place Jussieu \\  75005 Paris, France}

\email{\href{mailto:tat-dat.to@imj-prg.fr}{tat-dat.to@imj-prg.fr}}
\urladdr{\href{https://sites.google.com/site/totatdatmath/home}{https://sites.google.com/site/totatdatmath/home}}

\date{\today}

\begin{abstract}
We consider degenerate Monge-Amp\`ere equations  on compact Hessian manifolds.
We establish compactness properties of the set of normalized
quasi-convex functions and show local and global comparison principles
for twisted Monge-Amp\`ere operators. 
We then use the Perron method to solve Monge-Amp\`ere equations  whose RHS involves an arbitrary probability measure,
generalizing works of Cheng-Yau, Delano\"e, Caffarelli-Viaclovsky and Hultgren-\"Onnheim.
The intrinsic approach we develop should be useful in deriving similar results
on mildly singular Hessian varieties, in line with the Strominger-Yau-Zaslow conjecture.
\end{abstract}

\maketitle

\section*{Introduction}

An affine manifold $(M,\nabla)$ is a manifold possessing a flat affine connection $\nabla$. 
Equivalently we may define an affine manifold as a manifold possessing a
topological  atlas $(U_i,x^i)$
such that the transition functions $x^i \circ (x^j)^{-1}$ are affine maps.

A Hessian manifold  $(M,\nabla,g)$
is an affine manifold with a Riemannian metric $g$ 
which can be locally expressed as $g=\nabla d \phi$,
where $\phi$ is a (locally defined) smooth convex function.
Flat Riemannian manifolds provide examples of Hessian manifolds,
many more can be found in \cite{Sh07}.

Compact Hessian manifolds (with mild singularities) are an import class of affine manifolds, which  play a central role in 
the study of 
maximal degenerations of polarized  Calabi-Yau varieties $(X_t,g_t)$. Indeed
the Strominger-Yau-Zaslow conjecture \cite{SYZ96,GW00,LYZ05,KS06} predicts that 
(suitably rescaled) the metric spaces $(X_t,g_t)$
converge in the Gromov-Hausdorff sense to a singular Hessian manifold $M$
with  the limiting metric satisfying  a real Monge-Amp\`ere equation on the smooth locus of $M$.

Real Monge-Amp\`ere equations on Hessian manifolds with smooth data  have been studied  by Cheng-Yau \cite{CY82}, 
Delanoe \cite{De89} and  Caffarelli-Viaclovsky \cite{CV01}
(see also \cite{PT20} for a parabolic approach). 
Given $0<f \in {\mathcal C}^{\alpha}(M)$, they have shown that the equation
$$
\det (g+\nabla du)=c f(x) \det g
$$
admits a unique solution $(u,c)$, where $c$ is a positive constant and 
$u \in {\mathcal C}^{2,\alpha}(M)$ is a normalized $g$-convex function,
i.e. $u+\phi$ is convex in an affine chart where $g=\nabla d\phi$.
They have similarly shown the existence of a unique
$g$-convex function $u$ such that
$$
\det (g+\nabla du)= f(x) e^u \det g.
$$

\smallskip

The purpose of this note is to extend these results
to the case where the right hand side measure
$\mu=f \det g \, dx$ can be arbitrarily degenerate.
A similar study has been done by 
Hultgren-\"Onnheim in \cite{HO19},
by using the fact that $M=\Omega/\Gamma$ is the 
quotient of a convex subset of $\R^n$
by  a subgroup of affine transformations,
and by developping a variational approach.
We develop here a more intrinsic approach,
which one should be able to use
when $M$ has mild singularities, in
line with the SYZ conjecture; it
provides results of a different nature 
when $M$ is not special.

\smallskip

Recall that, for $s \in  \R$,
 a {\bf $s$-density} on $M$ is a section of a  line bundle whose transition functions are 
 $\left|\det\frac{\partial x_\alpha}{\partial x_\beta}\right|^{-s}$, where $x_\alpha=(x_\alpha^1,\ldots , x_\alpha^n)$ are local affine 
 coordinates.
 A 1-density is thus a generalization of the notion of a {\it volume form}.
If $\psi$ is a smooth function, then
$$
    \det\left(\frac{\partial^2 \psi }{\partial x^i_\beta \partial x^j_\beta }\right)= \left|\det\frac{\partial x_\alpha}{\partial x_\beta}\right| ^2 \det\left(\frac{\partial^2 \psi }{  \partial x^i_\alpha x^j_\alpha }\right).
$$
Thus  $\det(g_{ij}+u_{ij})$ 
is a  $2$-density  and 
$$
M_{\rho,g}[u]:=\rho \det(g_{ij}+u_{ij})dx^1\wedge \ldots\wedge dx^n 
$$
 is a globally well-defined measure on $M$ whenever $\rho$ is a (-1)-density
 and  $u$ is a ${\mathcal C}^2$ $g$-convex function. 
This is  the   Monge-Amp\`ere measure  of $u$ with respect to $\rho$.
 
 Using that any $g$-convex function can be uniformly approximated by smooth
 $g$-convex functions, we extend the definition of
 $M_{\rho,g}[u]$ to the set $\K(M,g)$ of all $g$-convex functions.
 Our first main result is then the following:

\begin{thmA}
Let $\mu$ be  a probability measure on $M$ and $\rho$ a $(-1)$-density.
There exists a unique $g$-convex function $u \in \K(M,g)$ such that
$$
M_{\rho, g}[u]=e^u \mu.
$$
\end{thmA}

The proof uses the Perron method, once an appropriate comparison principle
(Theorem \ref{thm:pcp}) has been established.
One can solve similarly various degenerate equations of the type
$M_{\rho, g}[u]=e^{F(u,x)} \mu$, under minimal assumptions on $F$.
We only consider here the case $F(u,x)=\e u$, where $\e>0$. Letting
$\e$ decrease to zero
we show convergence of the corresponding solutions $u_{\e}$,
 establishing our second main result:

\begin{thmB}
Let $\mu$ be  a probability measure on $M$ and $\rho$ a $(-1)$-density.
There exists a unique 
constant $c>0$ and a
$g$-convex function $u \in \K_0(M,g)$ such that
$$
M_{\rho, g}[u]=c \mu.
$$
\end{thmB}

These results allows one to prove that any $g$-convex function is the uniform limit of
smooth strictly $g$-convex functions, a result that can also be proved directly 
(see Proposition \ref{prop:reg})
by using convolutions and gluing techniques, the same way Demailly approximates
quasi-plurisubharmonic functions in \cite{Dem92}.

\smallskip

We study topological properties of subsets of $g$-convex functions in section \ref{sec:quasi-convex}
(see Lemmata \ref{lem:c0} and \ref{lem:lip}), define the twisted  Monge-Amp\`ere
operator $M_{\rho,g}$ in section \ref{sec:MA} and establish there the comparison principle
(Theorem \ref{thm:pcp}). We then prove Theorem A and Theorem B in section \ref{sec:resolution}.

\begin{ackn} 
We thank Jakob Hultgren for useful discussions.
 The authors are   partially supported by the ANR under the "PIA" program
ANR-11-LABX-0040 (research project HERMETIC), and by
  the CNRS through the IEA project PLUTOCHE.
\end{ackn}

\section{Quasi-convex functions on Hessian manifolds} \label{sec:quasi-convex}

\subsection{Hessian manifolds}

\begin{definition}
An {\it affine manifold} is a differentiable manifold $M$ equipped with a flat, torsion-free connection $\nabla$.
\end{definition}

It is known that  a manifold $M$ is affine if and only if $M$ admits an affine atlas $(U_i,x^i)$, i.e.
a topological atlas  such that transition functions $x^i \circ (x^j)^{-1}$ are  in the affine group 
${\rm Aff}(n,\R)=\{ \Phi: x \in \mathbb{R}^n \mapsto  Ax+b \in \mathbb{R}^n \}$.  

An affine manifold $(M,\nabla)$ is called {\it special } if it admits a volume form which is covariant constant with respect to the connection $\nabla$. 
Alternatively $M$ is special if and only if it admits an affine atlas with transition functions in $SL(n,\mathbb{R}) \times \R$. 

\smallskip

The notion of convex function $u:M \rightarrow \R$ on an affine manifold $M$ is well defined,
requiring that is is convex in "affine coordinates". However the only global convex functions
are the constants if $M$ is compact, as follows from the maximum principle.
We are therefore going to consider the softer notion of quasi-convex functions,
by allowing for a negative but smooth contribution from the Hessian of $u$.
We shall measure the latter by comparing it with a Hessian metric:

\smallskip

\begin{definition}
A Riemannian metric $g$ on an affine manifold $(M,\nabla)$ is called a {\it Hessian metric} if $g$ can be locally expressed by $g=\nabla d\phi  $.
In this case  $(M,\nabla, g)$ is called a {\it Hessian manifold}.
\end{definition}

If $(U_i,x^i)$ is an affine atlas the metric $g$ is Hessian if $g=\nabla d\phi_i$ in $U_i$,
where $\phi_i:U_i \rightarrow \R$ are smooth strictly convex functions
such that $\phi_i-\phi_j$ is affine in $U_i \cap U_j$.

\smallskip

By analogy with the concept of K\"ahler class, 
one makes the following:

\begin{definition}
 Two Hessian metrics $g=\nabla d\phi_i$ and $\tilde{g}=\nabla d \tilde{\phi}_i$ 
are in the same class if $\tilde{\phi}_i-\phi_i=\tilde{\phi}_j-\phi_j=u$ is independent of $i,j,$ hence defines
a global function $u:M \rightarrow \R$.
 \end{definition}

Such a function $u$ is then called $g$-convex:

\begin{definition}
A $g$-convex function on $M$ is a continuous function $u \in {\mathcal C}^0(M,\R)$ such that
$\phi_i+u$  is convex in any open set $U_i$ of $M$,  where $g=\na d \phi_i$ in $U_i$. 
\end{definition}

The definition does not depend on the choice of local potentials for $g$: 
if $g=\nabla d \tilde{\phi}_i$, then $\tilde{\phi}_i-\phi_i$ is affine hence
$u+\phi_i$ is convex if and only if so is $u+\tilde{\phi}_i$.

\begin{example}
Assume $M$ is compact.
If $u:M \rightarrow \R$ is smooth, it follows from the compactness of $M$ that
$\e u$ is $g$-convex for all $0 < \e \leq \e_0$, where $\e_0>0$ depends on the ${\mathcal C}^2$-norm of $u$.
We provide in Example \ref{exa:dirac2} examples of non smooth $g$-convex functions.
\end{example}

\begin{definition}
We let $\mathcal{K}(M,g)$ denote the set of all $g$-convex functions.
\end{definition}

In the sequel we endow $\mathcal{K}(M,g)$ with the ${\mathcal C}^0$-topology.
Basic operations on convex functions extend to $g$-convex ones:
\begin{itemize}
\item if $u,v$ are $g$-convex then so are $\max(u,v)$ and $\log[e^{u}+e^v]$;
\item a normalized sum of $g$-convex functions is $g$-convex;
\item a uniform limit of $g$-convex functions is $g$-convex.
\end{itemize}

Here is  recipy to cook up extra $g$-convex functions from known ones:

\begin{lemma}
 If $u$ is $g$-convex and $\chi$ is convex with $0 \leq \chi' \leq 1$ then $\chi \circ u$ is $g$-convex.
\end{lemma}

\begin{proof}
The assertion follows from an elementary computation which we provide for the convenience of the reader.
If $(x_1,\ldots,x_n)$ denote local affine coordinates in some chart $U$, where $g=\nabla d \phi$
with $\phi:U \rightarrow \R$ convex, we know that
$$
Hess(u+\phi):=\left( \frac{\partial^2 u}{\partial x_i \partial x_j}+\frac{\partial^2 \phi}{\partial x_i \partial x_j}  \right)
$$
and $Hess(\phi)$ are both non-negative. The function $v=\chi \circ u$ satisfies
$$
Hess(v+\phi)=
\chi'' \circ u \cdot \left(  \, \frac{\partial u}{\partial x_i}\frac{\partial u}{\partial x_j}  \right)
+\chi' \circ u \cdot  Hess(u+\phi)
+(1-\chi' \circ u) \cdot  Hess(\phi).
$$
Since the matrix $\left(  \, \frac{\partial u}{\partial x_i}\frac{\partial u}{\partial x_j}  \right)$ is non-negative,
the positivity of $Hess(v+\phi)$ follows from 
that of $Hess(\phi),Hess(u+\phi)$ and
the normalization $\chi'' \geq 0$ and $0 \leq \chi' \leq 1$.

One can interpret this computation in the sense of distributions, or alternatively
use convolutions in affine charts and proceed by approximation (see  Proposition \ref{prop:reg}).
\end{proof}

\subsection{Compactness properties of $g$-convex functions}

Let $(M, \nabla, g)$ be a compact Hessian manifold.
In the sequel, fixing a Hessian class $\{g\}$, we seek for 
$g$-convex functions that solve certain (degenerate)   Monge-Amp\`ere equations.

\smallskip

We shall use the Perron method and proceed by approximation, this requires to
establish good topological properties of families of $g$-convex functions.
Sup-normalized $g$-convex functions enjoy strong compactness properties:

\begin{lemma} \label{lem:c0}
The set 
$$
\mathcal{K}_0(M,g):=\{ u \in \mathcal{K}(M,g), \; \sup_M u=0 \}
$$
is compact. There exists $C_0 \in \R^+$ such that for all $u \in \mathcal{K}_0(M,g)$
and all $x \in X$,
$$
-C_0 \leq u(x) \leq 0.
$$
\end{lemma}

This result  has been established by Hultgren-\"Onnheim in \cite[Proposition 3.4]{HO19} by 
using properties of the universal cover of $M$ and the fact that convex functions admit
a Taylor expansion at order two at almost every point.
We provide here a different and direct approach, that only relies on submean value inequalities.

\begin{proof}
The closedness of $\mathcal{K}_0(M,g)$ is clear, as convexity and $\sup$-normalization are both preserved under 
uniform convergence. We show herebelow that
functions in $\mathcal{K}_0(M,g)$ are uniformly bounded and use this information to establish, in Lemma \ref{lem:lip},
that they are uniformly Lipschitz. It follows therefore from Arzela-Ascoli theorem that 
$\mathcal{K}_0(M,g)$ is compact in ${\mathcal C}^0(M,\R)$.

\smallskip

Fix  $u \in \mathcal{K}_0(M,g)$ and fix
two coverings $\{ U_\a \}, \{ U'_\a \}$ of $M$ by open sets such that
\begin{itemize}
\item there exists affine coordinates $(x_i^{\a})$ in  $U_\a'$;
\item  $g$ admits a smooth convex potential $\rho_\a$ in $U_\a'$;
\item $U_\a$ is relatively compact in $U'_\a$, hence $\rho_\a$ is uniformly bounded in $U_\a$.
\end{itemize}  
Pick $B(a,R)$ an affine ball in $U_\a$ and $\ell$ and affine line passing through $a$. It intersects
$\partial B(a,R)$ in two points $b^+,b^-$. The  function $v_\a=u+\rho_\a$
is convex in $U_\a'$ hence
$$
v_\a(x) \leq  \frac{1}{2R} \int_{b^-}^{b^+} v_\a(t) dt
$$
Using spherical coordinates and letting $\ell$ vary, we obtain
$$
v_\a(x) \leq  \frac{1}{{\rm Vol}(B(x,R))} \int_{B(x,R)} v_\a dV.
$$
Using that $\rho_\a$ is uniformly bounded in $U_\a$, we infer that for all $B(x,R) \subset U_\a$,
\begin{equation} \label{eq:l1}
u(x) \leq  \frac{1}{{\rm Vol}(B(x,R))} \int_{B(x,R)} u dV+C_{\a},
\end{equation}
for some uniform constant $C_\a >0$ (independent of $u,x,R$).
This uniform lower bound is the key to more general uniform
$L^1$ and $L^{\infty}$-bounds.

\medskip

We first establish a uniform $L^{1}$-bound, i.e. we claim there exists 
$C_1>0$ such that for all $u \in \mathcal{K}_0(M,g)$
$$
-C_1 \leq \int_{M} u dV \leq 0.
$$
Reasoning by contradiction, we assume there exists
$u_k \in  \mathcal{K}_0(M,g)$ such that
$\int_M u_k dV \rightarrow -\infty$.
Extracting and relabelling, we can assume that
$\sup_M u_k=u_k(x_k)=0$ with $x_k \rightarrow a \in U_\a$, for some $\a$.
We denote by $G$ the set of those $x \in M$ such that there exists 
a neighborhood $W$ of $x$ and a constant $C_W$ such that
$\int_W u_k dV \geq -C_W$ for all $k$.
The set $G$ is open by definition. 
It is non empty as it contains the point a: indeed
$$
0=u_k(x_k) \leq  \frac{1}{{\rm Vol}(B(x_k,R)} \int_{B(x_k,R)} u_k dV+C_{\a}
\leq   \frac{1}{{\rm Vol}(B(a,2R)} \int_{B(a,R/2)} u_k dV+C_{\a},
$$
as follows from (\ref{eq:l1}), the inclusions $B(a,R/2) \subset B(x_k,R) \subset B(a,2R)$ and   $u_k \leq 0$.

We finally claim that $G$ is closed  reaching a contradiction since then $G=M$ by connectedness.
Indeed assume that $(a_j) \in G^{\N}$  converges to $b \in M$. 
Fix $\a$ such that $b \in U_\a$ and $R>0$ small enough so that $B(b,R) \subset U_\a$.
For $j_0$ large enough $a_{j_0} \in B(b,R/2)$ and we can find a neighborhood
$W$ of $a_{j_0}$ such that $\int_{W} u_k dV \geq -C_W$; in particular
$k \mapsto \sup_{B(b,R/2)} u_k$ remains bounded.
It follows therefore from (\ref{eq:l1}) again that $k \mapsto \int_{B(b,R/2)} u_k dV$ is bounded, hence $b \in G$, as claimed.

\medskip

We now establish a uniform $L^{\infty}$-bound.
It is a consequence of the previous $L^1$-bound together with the following observation:
there exists $A,B>0$ such that for 
all $u \in  \mathcal{K}(M,g)$, 
$$
A \frac{\int_M u dV}{{\rm vol}(M)}-B \leq \inf_M u.
$$
The latter can be obtained as follows: 
fix $u \in  \mathcal{K}_0(M,g)$ and $b$ such that $u(b)=\inf_M u$.
We fix $\a$ and $r>0$ (independent of $u$) such that
$b \in U_{\a}$ and $B(b,4r) \subset U_\a$.
Fix $a \in U_\a$ with $d(a,b)=r$ and $c \in B(a,r/2)$.  
The affine line joining $b$ to $c$ meets $\partial B(c,d(c,b))$
in a second point $b'$. It follows from the convexity of
$v_\a=u+\rho_\a$ that
$$
v_\a(c) \leq \frac{v_\a(b)+v_\a(b')}{2} \leq \frac{v_\a(b)}{2}+C_\a.
$$
Integrating over $c \in B(a,r/2)$ and using that $\rho_\a$ is uniformly bounded yields the conclusion.
\end{proof}

We now observe that normalized $g$-convex functions are uniformly Lipschitz:

\begin{lemma} \label{lem:lip}
There exists $C_1>0$ such that for all $u \in \mathcal{K}_0(M,g)$
and all $x,x' \in X$,
$$
|u(x)-u(x')| \leq C_1 d_g(x,x').
$$
\end{lemma}

Here $d_g$ denotes the Riemannian distance induced on $M$ by $g$.
We refer the reader to \cite[Proposition 3.5]{HO19} for a related result.

\begin{proof}
We use the same notations as in the proof of the previous lemma. 
The result is a simple consequence of the uniform bound if $x,x'$ do not belong to the same
chart $U_\a$, so it suffices to treat this case. 

Observing that 
$$
|u(x)-u(x')| \leq \left| [u+\rho_\a](x)-[u+\rho_\a](x') \right| +| \rho_\a(x) -\rho_\a(x')|,
$$
we are reduced to establishing an appropriate result for (euclidean) convex functions. 
The latter follows from the following property:
if $v:\R^n \rightarrow \R$ is a convex function in a ball $B(x_0,2r)$
such that $m \leq v \leq M$, then for all $x,x' \in B(x_0,\delta)$,
$$
|v(x)-v(x')| \leq \frac{M-m}{\delta} ||x-x'||.
$$
The proof of this fact is left to the reader.
\end{proof}

\subsection{Regularization of $g$-convex functions}

We fix here again $(M, \nabla, g)$ a compact Hessian manifold of dimension $n$.

\begin{proposition}\label{prop:reg}
Let $u\in \mathcal{K}(M,g)$, there is a   sequence $u_j \in \mathcal{K}(M,g) \cap {\mathcal C}^{\infty}(M)$ 
such that $u_j$ uniformly converges to $u$ as $j\rightarrow \infty$. 
\end{proposition}

The proposition and its proof are inspired by Demailly’s regularization theorem \cite{Dem92}
(see also \cite{BK07})
for quasi-plurisubharmonic functions in complex geometry:
we use convolutions  in local charts,  affine  transitions  and gluing techniques  to construct global smooth $g$-convex approximants. 

\smallskip

We first recall the standard regularization by convolution. 
Let $\rho(x):=  \tilde \rho(|x|)\in C^\infty_0(\mathbb{R}^n)$ be a radial function with 
$\tilde \rho \geq 0, \rho(r)=0, \forall r\geq 1, \int_{\mathbb{R}^n } \rho d\lambda(x)=1$, where $d\lambda$ is the Lebesgue measure on $\mathbb{R}^n$. Set $\rho_\delta= \delta^{-n} \rho(x/\delta)$ 
and consider   
$$
 u_\delta(x)= \int_{\mathbb{R}^n} u(x-\delta w)\rho (w)d\lambda(w),
 $$
for $x \in \Omega' \Subset \Omega$ and $0\leq \delta\leq dist(\Omega',  \partial \Omega)$. If $u\in C^{0,\alpha}(\Omega)$, then $u_\delta\in C^\infty(\Omega')$ and 
$$\| u_\delta-u \|_{L^\infty(\Omega')}\leq \| u \|_{C^{0,\alpha}(\Omega)} \delta^\alpha.$$
In particular when $u$ is convex we can take $\alpha=1$.

\smallskip

We let $\max_\epsilon: \mathbb{R}^N\rightarrow \mathbb{R}$ denote the regularization 
 of the  $\max$ function,
\begin{equation}\label{eq:max_reg}
    \max{}_\epsilon(t_1,\ldots,t_N):=\int_{\mathbb{R}^N} \max (t_1+s_1, \ldots, t_N+s_N)\epsilon^{-N} \prod_{i=1}^N\gamma(s_i/\epsilon)ds_1\ldots ds_N,
\end{equation}
where $\gamma\in C^\infty(\mathbb{R}, \mathbb{R}^+)$ has compact support in $[-1,1]$ 
and is such that $\int_{\mathbb{R}}\gamma(s)ds=1$ with$\int_{\mathbb{R}}s\gamma(s)ds=0$.
 It follows from the definition that $\max_\epsilon$ is non-decreasing in all variables, smooth and convex on $\mathbb{R}^N$. 
 
 \smallskip

The following lemma is left to the reader:

\begin{lemma}\label{lem:trans}
Let $A: U\rightarrow U'$ be an affine map between two open subsets $U, U'$ of $\mathbb{R}^n$,
and let $u$ be a convex function on $U$. Then for any set $V\Subset U$ we can find a constant $\delta_V>0$ such that for any $\delta\in (0,\delta_V)$ the function $u^A:= (u\circ A^{-1})_\delta\circ A$ is 
convex and well defined 
in a neighborhood of $\overline V$ and there  is a constant $C_V>0$ such that 
$$ 
\|u^A_\delta -u\|_{L^\infty(V)}\leq C_V\| u\|_{C^{0,1}}\delta.
$$
\end{lemma}

\smallskip

\begin{proof}[Proof of Proposition  \ref{prop:reg}]
Let $(U_i)_{i\in I}$ be a finite cover with local affine charts  of $M$  and choose another finite cover with local affine charts $(V_i)_{i\in I}$ of $M$ such that $V_i\Subset U_i$. For each $i\in I$ we can find a convex function $\phi_i$ in a neighborhood $W_i$ of $\overline U_i$ such that $g=D^2 \phi_i$ on $W_i$. Then the function
$
v_i:= \phi_i+ u 
$
is convex  on $U_i$. 

For any pair $(j,k)\in I^2$ such that $U_j\cap U_k\neq \emptyset$, we  have two   regularizations $v_{j, \epsilon}$ and $v_{k,\epsilon}$ of  the restriction $v_j|_{U_j\cap U_k}$ using convolutions on local charts $U_j$ and $U_k$ respectively.  Let $A$ be the affine change of coordinates on $U_j\cap U_k$ from $U_j$ to $U_k$.  Then we have on $U_j\cap U_k$
$$v_{j,\epsilon }-v_{k,\epsilon} = v_{j,\epsilon} -v^A_{j,\epsilon}+ (v_j-v_k)_\epsilon, $$
where $(v_j-v_k)_\epsilon$ is the regularization of $u_j-u_k$ using convolution on $U_k$. Using  Lemma \ref{lem:trans} and the fact that $v_j-v_k= \phi_j-\phi_k\in C^\infty$, we obtain
$$
\|v_{j,\epsilon }-v_{k,\epsilon} -(\phi_j-\phi_k) \|_{L^\infty}\leq B\epsilon \quad \text{on} \quad U_j\cap U_k.
$$

Fix $C_1>>1$. We define for each $i\in I$ a smooth function $\eta_i$ on $U_i$ such that $\eta_i=0 $ on $V_i$ and $\eta_i=-C_1$ away from a
compact subset if $U_i$.  Suppose that $D^2 \eta_j\geq - C_2 g$ for some $C_2>0$. We define the function
$$w_{j}^\epsilon = v_{j, \epsilon}-\phi_j+  B\epsilon \eta_j \quad \text{on} \quad U_j.$$

Then  
$\varphi_\epsilon (x) =  \max{}_{\epsilon} \{w^\epsilon_j(x): x\in U_j\}$
 is smooth and  $(1+C_2B\epsilon)g$-convex. 
We infer that  $u_\epsilon:=\varphi_\epsilon/( 1+C_2B\epsilon) \in \mathcal{K}(M,g)$ uniformly converges  to $u$ as $\epsilon\rightarrow 0$. 
\end{proof}

\section{Monge-Amp\`ere operators on compact Hessian manifolds} \label{sec:MA}

\subsection{Definition of Monge-Amp\`ere operators}

\subsubsection{Alexandrov definition}

Let $\Omega$ be an open subset of $\R^n$ and $u: \Omega\rightarrow \R$ be a  convex function. 
\begin{definition}
The subdifferential of $u$ is the set-valued function $\partial u: \Omega\rightarrow \mathcal{P}(\R^n)$ defined by
\begin{equation}
    \partial u(x_0) =\{p \in \R^n, \; u(x)\geq u(x_0)+p.(x-x_0), \forall x\in \Omega \}.
\end{equation}
Given $E\subset \Omega$, we define $\partial u(E)=\cup_{x\in E} \partial u(x)$. 
\end{definition}

It follows from a theorem of Alexandrov  that 
$$
\mathcal{S}:= \{E\subset \Omega| \,  \partial u(E) \text{ is Lebesgue measurable}\}
$$
 is a $\sigma$-algebra (cf. \cite[Chapter 1]{Gut}). This motivates the following:
 
\begin{definition}[Alexandrov]
The Monge Amp\`ere measure, $\M u $,  of a convex function on $\Omega$ is defined by
\begin{equation}
    \M u(E)= |\partial u(E)| 
\end{equation}
for any Borel set $E\subset \Omega$.
\end{definition}

Here $|B|$ denotes the Lebesgue measure of the Borel set $B$.
For smooth convex functions 
one can check that this definition yields
$$
 \M u(E)=\int_E \det D^2 u(x) dx
$$  
(see \cite[Example 1.1.4]{Gut}).

\begin{example} \label{exa:dirac1}
The convex function $u:x \in \R^n \mapsto |x-a| \in \R$ 
 is smooth  off the point $a$ and affine along lines through  $a$, thus 
 $\det D^2 u=0$ in $\R^n \setminus \{a\}$.
On the other hand $\partial u(a)=B(0,1)$, 
therefore $\M u=\delta_a$ is the Dirac mass at the point $a$.
\end{example}

We shall use the following basic results:

\begin{lemma}\label{lem:mass_com}\cite[Lemma 1.4.1]{Gut}
Let $\Omega\subset \mathbb{R}^n$ be a bounded open set and $u,v\in C(\bar \Omega)$.
 If $u=v$ on $\partial \Omega$ and $u\leq v$ in $\Omega$, then
$\M[v](\Omega)\leq \M [u](\Omega)$.  
\end{lemma}

\begin{lemma}\label{lem:convex}\cite[Lemma 1.4.7]{Gut}
If $u$ and $v$ are convex functions in $\Omega$, then
\begin{equation}
\M[u+v](E)\geq \M [u](E)+\M[v](E)
\end{equation}
for any Borel set $E\subset\Omega$.
\end{lemma}

\subsubsection{The Rauch-Taylor point of view}

In \cite{RT} the authors introduce an alternative  definition of the Monge-Amp\`ere measure:
for a smooth $u\in \K(\Omega)$ they observe that
$$
\mathcal{M}u:= du_1\wedge \cdots\wedge du_n=\det \left( \frac{\partial^2 u }{\partial x_i \partial x_j} \right)
dx_1 \wedge \cdots \wedge dx_n,
$$
setting $u_j:=\frac{\partial u }{\partial x_j}$. 
The  mass of the Monge-Amp\`ere measure $d u_1 \wedge \cdots\wedge du_n$ is controlled by a Chern-Levine-Nirenberg inequality:

\begin{lemma}\label{lem:CLN}
Let $\Omega$ be a subset in $ \R^n$ and fix $\Omega_1\Subset \Omega_2\Subset \Omega$  relatively compact open subsets.
 Let $u$ be a ${\mathcal C}^2$ convex function and $T=\phi dx^{k+1}\wedge \ldots\wedge dx^n$ for some positive continuous function $\phi$. There exists a constant $C=C_{\Omega_1,\Omega_2}>0$ such that  
$$
\int_{\Omega_1} d u_1\wedge \ldots \wedge d u_k \wedge T \leq C \| u\|^k_{L^\infty(E)} \|\phi\|_{L^\infty (\Omega_1)},
$$
where $E= (\Omega_2\setminus \Omega_1)\cap Supp(\phi)$ and $u_j$ denotes $\frac{\partial u}{\partial x^j}$. 
\end{lemma}

We provide a  proof for the reader's convenience.

\begin{proof}
The argument is similar to  the (complex) Chern-Levine-Nirenberg inequality
(see \cite[Theorem 3.9]{GZbook}).
 By induction it suffices to prove the inequality for $k=1$.

Let $\chi $ be a non-negative smooth  function on $\Omega$ such that $\chi=1$ in $\Omega_1$. 
Since $u$ is a ${\mathcal C}^2$ convex function, we have
$$
d u_1\wedge \ldots \wedge d u_k \wedge dx^{k+1}\wedge \ldots\wedge dx^n = 
\det[(u_{ij})_{1\leq i,j\leq k}] dx^{1}\wedge \ldots\wedge dx^n \geq 0.
$$
For $T=\phi dx^{2}\wedge \ldots\wedge dx^n$ we thus get
$$
\int_{\Omega_1} du_1\wedge T \leq 
\|\phi\|_{L^\infty (\Omega_1)}\int_{\Omega_2} \chi d  u_1\wedge dx^2\wedge\ldots\wedge dx^n.
$$
Using Stokes theorem   we obtain
\begin{eqnarray*}
\int_{\Omega_2} \chi d u_1\wedge  dx^2\wedge\ldots\wedge dx^n &=& -\int_{\Omega_2}  u_1 d\chi \wedge  dx^2\wedge\ldots\wedge dx^n\\
&=&  \int_{\Omega_2} u d\chi_1 \wedge  dx^2\wedge\ldots\wedge dx^n\\
&=& \int_{\Omega_2\setminus \Omega_1} u d\chi_1 \wedge  dx^2\wedge\ldots\wedge dx^n.
\end{eqnarray*}
Fixing   $C>0$ such that $(\chi_{ij}) \leq C I_n$, we infer
$\int_{\Omega_1} du_1\wedge T\leq  C \| u\|_{L^\infty(E)} \|\phi\|_{L^\infty (\Omega_1)}$. 
\end{proof}

If $u_j$ are smooth convex functions uniformly converging to $u$, the measures 
$\mathcal{M}u_j$ have uniformly bounded masses thanks to Lemma \ref{lem:CLN}
and one can show that they converge towards a measure $\mathcal{M}u$
independent of the approximants.
Thus $\mathcal{M}u$ is well-defined for arbitrary convex functions
and one can check that the two definitions are equivalent:

\begin{proposition}
If $u $ is convex in $\Omega$ then $\M u = \mathcal{M} u$.
\end{proposition}

 We refer the reader to \cite[Proposition 3.3]{RT} for a proof.

\subsubsection{The compact case}

Let $(M, \nabla, g)$ be a compact Hessian manifold of dimension $n$.

Recall that, for $s \in  \R$,
 a {\bf $s$-density} on $M$ is a section of a  line bundle whose transition functions are $\left|\det\frac{\partial x_\alpha}{\partial x_\beta}\right|^{-s}$, where $x_\alpha=(x_\alpha^1,\ldots , x_\alpha^n)$ are local coordinates in some open set $U_\alpha$ of $M$. 
 A 1-density is thus a generalization of the notion of a {\it volume form}.

Observe that if  $x_\alpha=(x_\alpha^1,\ldots , x_\alpha^n)$ are   affine local coordinates with respect to $\nabla$
and $\psi$ is a smooth function,  then
$$
    \det\left(\frac{\partial^2 \psi }{\partial x^i_\beta \partial x^j_\beta }\right)= \left|\det\frac{\partial x_\alpha}{\partial x_\beta}\right| ^2 \det\left(\frac{\partial^2 \psi }{  \partial x^i_\alpha x^j_\alpha }\right).
$$
Thus  $\det(g_{ij}+u_{ij})$ 
is a  $2$-density 
and $\det(g_{ij}+u_{ij})dx^1\wedge \ldots\wedge dx^n $ is not a well-defined measure on $M$. 
 To obtain a good definition 
 we use a (-1)-density:

\begin{definition}
Let $\rho$ be an (-1)-density of $M$ and assume that $u$ is a ${\mathcal C}^2$ $g$-convex function. 
We define the (relative) Monge-Amp\`ere measure  of $u$ with respect to $\rho$  by
$$
M_{\rho,g}[u]:=\rho \det(g+u_{ij})dx^1\wedge \ldots\wedge dx^n.
$$
\end{definition}

\medskip

Observe that
$
   \rho_{\beta}  \det\left(g+\frac{\partial^2 u }{\partial x^i_\beta \partial x^j_\beta }\right)= 
 \rho_{\alpha}     \left|\det\frac{\partial x_\alpha}{\partial x_\beta}\right|  \det\left(g+\frac{\partial^2 u }{  \partial x^i_\alpha x^j_\alpha }\right),
$
thus the measure $M_{\rho,g}[u]$ does not depend on the choice of affine coordinates, 
it is  a globally well defined  Radon measure on $M$.
We simply denote this measure by $M_{g}[u]$ when $\rho=\det(g)^{-\frac{1}{2}}$.

\smallskip

If $u_j \in {\mathcal K}(M,g) \cap {\mathcal C}^2(M)$ uniformly converges to $u$, it follows from Lemma \ref{lem:CLN}
that the measures $M_{\rho,g}[u_j]$ have uniformly bounded masses
$\int_M M_{\rho,g}[u_j] \leq C_1$.
We can extend the definition of  the Monge-Amp\`ere measure to arbitrary $g$-convex functions
by approximation,   following the method of Rauch-Taylor \cite{RT}:

\begin{proposition} \label{pro:limunif}
Assume $u_j, v_j\in \mathcal{K}(M,g)\cap C^2(M)$   are such that 
$\lim_{j\rightarrow \infty} {u_j} = \lim_{j\rightarrow \infty} {v_j}=u\in \mathcal{K}(M,g)$,
$ M_{\rho,g}[u_j] \rightarrow \mu$ and 
$ M_{\rho,g}[v_j] \rightarrow \nu$  in the topology of  weak convergence of measures. Then $\mu=\nu$.
\end{proposition}

Using Proposition \ref{prop:reg} we can thus set:

\begin{definition}
If $u$ is an arbitrary $g$-convex function, we  define
 $$
 M_{\rho,g}[u]=\lim_{j \rightarrow +\infty} M_{\rho,g}[u_j],
 $$ 
where $u_j$ is {\it any}  sequence of smooth $g$-convex functions converging to $u$.
\end{definition}

\begin{example} \label{exa:dirac2}
Let $x$ be  affine coordinates near $a \in M$
and $\chi$ a  test function such that $\chi \equiv 1$ near $a$
with compact support in this  chart.
Then $u:x  \in M \mapsto \chi(x) |x-a| \in \R$ is smooth in $M \setminus \{a\}$
and convex near $a$, so $\e u$ is $g$-convex if $0<\e$ is small enough.

The Monge-Amp\`ere measure $M_{\rho,g}[\e u]$ satisfies
$M_{\rho,g}[\e u] \geq C \e^n \delta_a=e^{u(a)+n \log \e+\log C} \delta_a$,
so the $g$-convex function $v=\e u+n \log \e+\log C$  satisfies
$M_{\rho,g}[v] \geq e^v \delta_a$: one says that
$v$ is a subsolution to the equation $M_{\rho,g}[w] =e^w \delta_a$.
Taking finite convex combination of such functions and shifting by an additive constant, one can construct similarly
$g$-convex subsolutions to the equation
$M_{\rho,g}[w] =e^w \mu$, where
$\mu=\sum_{i=1}^p c_i \delta_{a_i}$.
\end{example}

It follows from Stokes theorem that the Monge-Amp\`ere measures $M_{\rho,g}(u)$ all have the same total mass
if the manifold $M$ is special. This is no longer the case on an arbitrary
 Hessian manifold, 
but we nevertheless have uniform bounds :

\begin{lemma} \label{lem:masseMA}
There exist constants $0<a \leq b <+\infty$ such that  
$$
0<a \leq \inf_{u \in {\mathcal K}(M,g)} \int_M M_{\rho,g}(u) \leq \sup_{u \in {\mathcal K}(M,g)} \int_M M_{\rho,g}(u)  \leq b.
$$
\end{lemma}

\begin{proof}
Since $M_{\rho,g}(u)=M_{\rho,g}(u-\sup_M u)$, we can consider the infimum and the supremum over the set 
${\mathcal K}_0(M,g)$ which is a compact subset of ${\mathcal C}^0(M,\R)$ (see Lemma \ref{lem:lip}).

Since the map $u \mapsto m(u)=\int_M M_{\rho,g}(u)$ is continuous and takes finite values,
 it suffices to check that $m(u)>0$ for all $u \in {\mathcal K}(M,g)$.
 This is an easy consequence of Theorem \ref{thm:pcp} below:
 if $M_{\rho,g}(u)=0$ then for any constant function $A \in {\mathcal K}(M,g)$, one gets
 $$
0= e^{-u} M_{\rho,g}(u) \leq e^{-A}M_{\rho,g}(A),
 $$
 hence $u \geq A$, which leads to a contradiction as soon as $A>\sup_M u$.
\end{proof}

\subsection{Comparison principles}

We first establish a local comparison principle.

\begin{lemma}\label{lem:local_com}
Let $u$,  $v$ be two convex functions on $\Omega$ such that 
\begin{equation}\label{eq:cond}
e^{-u}\M[u]\geq e^{-v} \M[v].
\end{equation}
and 
$u-v$ achieves a local  strict maximum at $x_0\in \Omega$. 
Then $u(x_0)\leq v(x_0)$. 
\end{lemma}

\begin{proof}
It follows from the hypothesis that there is $D\subset \Omega$ such that $ \sup_{\partial D} (u-v)=b$ and 
$\sup_{D}(u-v)=u(x_0)-v(x_0)=a$ with $a>b$. Assume by contradiction that $a>0$.
Shrinking $D$ we can assume that $b\geq 0$. Define
$$
u_\delta=u+\delta\|x-x_0\|^2 -\frac{a+b}{2},
$$
with $\delta>0$ satisfying $\delta (diam(D))^2< (a-b)/2\leq  (a+b)/2 $.  Therefore we have $u_\delta\leq u$ in $D$. 
The open set $B= \{x\in D: u_\delta> v\}$ is non empty as it contains $x_0$. For any $x\in \partial D$ we have
$u_\delta(x)< v(x)$, therefore $\partial B= \{x\in D: u_\delta
= v\}$. 

It follows from Lemma \ref{lem:mass_com} that 
\begin{equation} \label{eq:inq}
{\rm M}[v] (B) \geq  {\rm M} [u_\delta]  (B) \geq {\rm M} [u](B)+ (2\delta)^n |B|.  
\end{equation}

On the other hand, using inequality \eqref{eq:cond} and the fact that  $v\leq u_\delta\leq u$ in $B$, we get
 $$
 {\rm M}[v] (B)   \leq  ( e^{v-u} {\rm M} [u])(B) \leq  {\rm M} [u] (B).
 $$ 
 Using \eqref{eq:inq} we infer $|B|=0$,  a contradiction . 
\end{proof}

We now prove a global comparison principle:

\begin{theorem} \label{thm:pcp}
Let $(M,g,\nabla)$ be a compact Hessian manifold and  $u$ and $v$ be two  $g$-convex functions such that
\begin{equation}\label{eq:cond_2}
e^{-u} M_{\rho, g}[u]\geq
e^{-v} M_{\rho,g}[v]
\end{equation}
in the sense of  measures. Then $ u\leq v$. 
\end{theorem}

\begin{proof}
Let $h$ be a smooth strictly $g $-convex function 
such that $g+\nabla d h\geq \epsilon g$ and $h\leq u$. For any $0<\delta<1$, we consider $ u_\delta=(1-\delta )u+ \delta h+n\log(1-\delta)$. 
Our goal is show that $u_\delta \leq v$ for all $0<\delta<1$.

Suppose that $\max_M(u_\delta-v) =  u_\delta(x_\delta)-v(x_\delta)$.
We  take a  neighborhood $D$ of $x_\delta$ and consider 
$\tilde{u}_\delta=\phi+u_\delta$, $ \tilde v= \phi+v$ where $g=\nabla d  \phi$ in $D$.  Define 
$$
\tilde u_{\alpha}=\tilde{u}_\delta -\alpha \|x-x_\delta\|^2 =(1-\delta)(\phi+u)+\delta(\phi+h)-\alpha \|x-x_\delta\|^2  +n\log (1-\delta) 
$$ 
with $\alpha >0$ such that $\delta(\phi+h)  -\alpha \|x-x_\delta\|^2 $ is convex (we use here the fact that $h$ is strictly $g$-convex). 
Then $\tilde u_\alpha -\tilde v $  achieves a strict maximum at $x_\delta$ on $D$.  Moreover in $D$ we also have
\begin{eqnarray*}
e^{-\tilde u_\alpha} {\rm M}[ \tilde{u}_\alpha]&\geq& e^{ -\tilde{u}_\alpha}(1-\delta)^n {\rm M}[u+\phi] \\
&\geq &e^{-u-\phi}  {\rm M}[ u+\phi]\\
&\geq& e^{-v-\phi}  {\rm M}[v+\phi] = e^{-\tilde v}  {\rm M}[\tilde v], 
\end{eqnarray*}
where we use \eqref{eq:cond_2}
for the third  inequality. 
Applying Lemma \ref{lem:local_com} to $\tilde u_\alpha $ and $\tilde{v}$ we get $\tilde{u}_\alpha(x_\delta)-\tilde{v}(x_\delta)\leq 0$. 
Letting $\alpha\rightarrow 0$, we obtain $u_\delta\leq v$ on $M$. Letting $\delta\rightarrow 0$ we get $u\leq v$ on $M$ as required. 
\end{proof}

The maximum of two $g$-convex functions is also $g$-convex.
The following inequality  allows one to bound from below
the corresponding Monge-Amp\`ere measure:

\begin{lemma} \label{lem:MAmax}
Let $u,v$ be two $g-$convex functions then
$$M_{\rho,g}[\max (u,v)]\geq 1_{\{u\geq v\}} M_{\rho,g}[u] + 1_{\{u< v\}} M_{\rho,g}[v]. $$
\end{lemma}

\begin{proof}
The set $\Omega :=\{u<v\}$ is an open set of $M$ since $u$ and $v$ are continuous, hence
$$
1_{\{u< v\}} M_{\rho,g}[\max (u,v)]= 1_{\{u< v\}} M_{\rho,g}[v].
$$
We infer that
$M_{\rho,g}[\max (u,v)]\geq 1_{\{u > v\}} M_{\rho,g}[u] + 1_{\{u< v\}} M_{\rho,g}[v].$
Thus we are done if $\mu(\{u=v\})=0$ with $\mu = M_{\rho,g}[u]$. 

\smallskip

We claim that $\mu(\{u=v+\epsilon\})=0$ for all $\epsilon \in \mathbb{R}\setminus S_\mu$ where $S_\mu$ is at most countable. 
Assuming this we can find a  sequence $\epsilon_j$   which converge to $0$ such that $\mu(\{u=v+\epsilon_j\})=0$.
 Replacing $v$ by $v+\epsilon_j$ the argument yields
  $$
  M_{\rho,g}[\max (u,v+\epsilon_j)]\geq 1_{\{u \geq  v+\epsilon_j\}} M_{\rho,g}[u] + 1_{\{u< v+\epsilon_j\}} M_{\rho,g}[v],
  $$
we obtain the desired inequality by letting $\epsilon_j\rightarrow 0$ and using Proposition \ref{pro:limunif}.

\smallskip

We now verify that the set $ \{ \epsilon \in \mathbb{R}: \mu(\{u=v+\epsilon\})>0 \}$ is at most countable. 
Observe that the function 
$f: t\in \mathbb{R}\rightarrow \mu(\{u<v+t \}) \in \mathbb{R}^+$
is increasing and left continuous since $\mu=M_{\rho,g}[u]$ is a Borel measure. Moreover 
$$
\lim_{t\rightarrow \epsilon^+}f(t) =\mu(\{u\leq v+\epsilon\}),
$$
hence $f$ is continuous at $\epsilon$ unless $\mu(\{u=v+\epsilon\})>0$. 
Therefore  $S_\mu$ is the set of discontinuity of $f$, hence $S_\mu$ is at most countable.  
\end{proof}

\section{Resolution of Monge-Amp\`ere equations} \label{sec:resolution}

Let $(M, \nabla, g)$ be a compact Hessian manifold of dimension $n$.
In this section we prove Theorem A and Theorem B from the introduction.

\subsection{Perron method}

\begin{theorem} \label{thm:solutionMA}
Let $\mu$ be  a probability measure on $M$ and $\rho$ a $(-1)$-density.
There exists a unique $g$-convex function $u \in \K(M,g)$ such that
$$
M_{\rho, g}[u]=e^u \mu.
$$

Moreover if $\mu_j$ are probability measures that weakly converge to a probability measure $\mu$,
then the unique solutions $u_j \in \K(M,g)$ to $M_{\rho, g}[u_j]=e^{u_j} \mu_j$ uniformly converge
to the unique solution $u \in \K(M,g)$ of $M_{\rho, g}[u]=e^u \mu$.
\end{theorem}

\begin{proof}
We are going to apply the Perron method, showing that the envelope of subsolutions is the unique solution
to this equation. 

\smallskip

\noindent {\it Step 1.}
We start by treating the case when $\mu=\sum_{i=1}^p c_i \delta_{a_i}$ is a sum of Dirac masses,
where $p \in \N^*$, $c_1,\ldots,c_p>0$ are positive reals such that $\sum_{i=1}^p c_i=1$,
and $a_1,\ldots,a_p$ are distincts points in $X$.
We let ${\mathcal F}$ denote the family of subsolutions, i.e. 
$$
{\mathcal F}=\{ u \in  \K(M,g), \; M_{\rho, g}[u] \geq e^u \mu \}.
$$
Here are basic properties of ${\mathcal F}$:
\begin{itemize}
\item it follows from Example \ref{exa:dirac2} that ${\mathcal F}$ is not empty, we pick $u_0 \in {\mathcal F}$;
\item we claim that ${\mathcal F}$ is uniformly bounded from above. 
Indeed fix $u \in {\mathcal F}$ and set $v=u-\sup_M u \in {\mathcal K}_0(M,g)$.
It follows from Lemma \ref{lem:c0} that $-C_0 \leq v(x)$ for all $x \in M$, hence in particular
$\sup_M u \leq u(a_1)+C_0$. 
Lemma \ref{lem:masseMA} now yields
$$
b \geq \int_M M_{\rho, g}[u] \geq \sum_{i=1}^p c_i e^{u(a_i)} \geq c_1 e^{u(a_1)}
$$
since $u$ is a subsolution, hence
$
\sup_M u \leq \log (b/c_1)+C_0.
$
\item ${\mathcal F}$ is closed and stable under maximum: if $u,v \in {\mathcal F}$, 
then
\begin{eqnarray*}
M_{\rho, g}[\max(u,v)] &\geq & 1_{\{u \geq v \}} M_{\rho, g}[u] +1_{\{u < v \}} M_{\rho, g}[v] \\
&\geq &  1_{\{u \geq v \}} e^u \mu  +1_{\{u < v \}} e^v \mu  \\
&=& 1_{\{u \geq v \}} e^{\max(u,v)} \mu  +1_{\{u < v \}} e^{\max(u,v)} \mu =e^{\max(u,v)} \mu,
\end{eqnarray*}
as follows from Lemma \ref{lem:MAmax}.
\end{itemize}
It follows therefore from Lemma \ref{lem:lip} and Arzela-Ascoli theorem that
 $$
 {\mathcal F}_{u_0}=\{ u \in  \K(M,g), \; M_{\rho, g}[u] \geq e^u \mu  \text{ and } u \geq u_0 \}.
 $$
 is a compact subset of $ \K(M,g)$.
Thus the envelope of subsolutions $U=\sup_{ {\mathcal F}} u=\sup_{ {\mathcal F}_{u_0}} u  $
  is a $g$-convex function which is still a subsolution.
  
 We finally conclude that $U$ is actually a solution of the equation through a balayage process.
We pick a small euclidean ball $B$ in some affine chart 
such that $\partial B$ does not contain any point $a_j$,
and we solve the local Dirichlet
problem $M_{\rho, g}(v)=e^v \mu$ in $B$ with $U$ as boundary data.
The local solution exists by (a slight generalization of) \cite[Theorem 1.6.2]{Gut}
and glue with $U$ in $M \setminus B$ to provide yet another subsolution.
It thus coincides with $U$, hence $U$ solves the equation in
any such ball $B$, hence in the whole of $M$.
The uniqueness follows from the comparison principle (Theorem \ref{thm:pcp}).

\medskip
\noindent {\it Step 2.}
We now proceed by approximation in order to treat the general case.
Let $\mu_p=\sum_{i=1}^p c_{i,p} \delta_{a_{i,p}}$ be finite combination of Dirac masses that
weakly approximate $\mu=\lim_{p \rightarrow +\infty} \mu_p$.
It follows from previous step that there exists a unique $g$-convex function
$u_p \in \K(M,g)$ such that
$M_{\rho, g}[u_p]=e^{u_p} \mu_p$.

We claim that $\sup_M u_p$  is uniformly bounded. 
We set $v_p=u_p-\sup_M u_p \in \K_0(M,g)$
and recall from Lemma \ref{lem:masseMA} that the Monge-Amp\`ere mass
of $M_{\rho,g}(u_p)=M_{\rho,g}(v_p)$ is uniformly bounded from above and below, away from zero.
We infer
 $$
 a \leq \int_M M_{\rho,g}(u_p) =\int_M e^{u_p} d\mu_p  \leq e^{\sup_M u_p}
 \Longrightarrow \log a \leq \sup_M u_p.
 $$
Using that $v_p \geq -C_0$ (see Lemma \ref{lem:c0}), we also obtain
  $$
e^{\sup_M u_p} e^{-C_0} \leq  \int_M e^{u_p} d\mu_p  =  \int_M M_{\rho,g}(u_p)  \leq b
 \Longrightarrow  \sup_M u_p \leq C_0+ \log b.
 $$

Lemma \ref{lem:c0} ensures that $u_p$ is uniformly bounded on $M$, while Lemma \ref{lem:lip}
ensures that the $u_p$'s are uniformly Lipschitz, hence relatively compact for the
${\mathcal C}^0$-topology. We can thus extract a convergent subsequence
$u_{p_j} \rightarrow u \in \K_0(M,g)$. Since 
$M_{\rho, g}$ is continuous for the ${\mathcal C}^0$-topology, we infer
 $$
 M_{\rho, g}[u]=e^u \mu.
 $$
 The uniqueness follows again from the  comparison principle (Theorem \ref{thm:pcp}).
 
 \medskip
\noindent {\it Step 3.}
We finally prove the stability property.
Let $\mu_j$ be probability measures that weakly converge to a probability measure $\mu$.
Let  $u_j \in \K(M,g)$ be the unique solutions to $M_{\rho, g}[u_j]=e^{u_j} \mu_j$ 
and let  $u \in \K(M,g)$ the unique solution of $M_{\rho, g}[u]=e^u \mu$.

The same reasoning as above shows that $\sup_M u_j$ is uniformly bounded, hence
$(u_j)$ is relatively compact. A subsequence $u_{j_k}$ thus uniformly converge
to some function $v \in \K(M,g)$. Now $ M_{\rho, g}[u_{j_k}] \rightarrow M_{\rho, g}[v]$
and $e^{u_{j_k}} \mu_{j_k} \rightarrow e^{v} \mu$, so $v=u$ by uniqueness
\end{proof}

\subsection{The flat equation}

As the proof of Theorem \ref{thm:solutionMA} shows, a similar result holds for more general equations of the form
$M_{\rho, g}[u]=F(x,u) \mu$, with appropriate assumptions on the function $F$.
A straightforward generalization that we shall need is that for any $\e>0$, there exists
a unique $g$-convex function $u_{\e} \in \K(M,g)$ such that
$$
M_{\rho, g}[u_{\e}]=e^{\e u_{\e}} \mu.
$$
We now use such perturbations   to solve  another degenerate Monge-Amp\`ere equation:

\begin{theorem} \label{thm:solutionMA2}
Let $\mu$ be  a probability measure on $M$ and $\rho$ a $(-1)$-density.
There exist a unique constant $c>0$ and  a $g$-convex function $u \in \K_0(M,g)$ such that
$$
M_{\rho, g}[u]=c \mu.
$$
\end{theorem}

When $M$ is special the preservation of Monge-Amp\`ere masses ensures
that the constant $c$ is determined by
$$
\int_M M_{\rho, g}(0)=c.
$$
The uniqueness of $c$ is slightly more involved in the general case.

\begin{proof}
{\it Existence of $(c,u)$}.
We first show the existence of  the solution.
Fix $\e>0$ and let  $u_{\e} \in \K(M,g)$ be the unique $g$-convex function such that
$$
M_{\rho, g}[u_{\e}]=e^{\e u_{\e}} \mu.
$$
It follows from Lemma \ref{lem:masseMA} that 
$$
a \leq \int_M M_{\rho, g}[u_{\e}]=\int_M e^{\e u_{\e}} \mu \leq  e^{\e \sup_M u_{\e}},
$$
hence $\e \mapsto \e \sup_M u_\e$ is uniformly bounded below.

Lemma \ref{lem:masseMA} again yields a bound from above
on $\int_M e^{\e u_\e} d\mu \leq b$. It follows therefore
from the concavity of the logarithm that
$$
\int_M \e u_\e d\mu \leq \log b.
$$
Since $\e u_\e$ is $g$-convex for all $0 \leq \e \leq 1$,  Lemma \ref{lem:c0}
ensures that $\int_M \e u_\e d\mu $ and $\e \sup_M u_\e$ are uniformly comparable, hence
$\e \mapsto \e \sup_M u_\e$ is uniformly bounded as $\e \searrow 0$.

The family $v_{\e} =\e u_{\e}$ is thus relatively compact in ${\mathcal K}(M,g)$
by Lemma \ref{lem:lip}, so we can extract a sequence 
$v_{\e_j}$ which uniformly converges to a $g$-convex function $v$.
Since $v_{\e_j}$ is actually $\e_j g$-convex, the function $v$ is 
$0g$-convex hence $v \equiv c$ is constant by the maximum principle.

The family $\e \mapsto w_\e=u_\e-\sup_M u_\e$  is relatively compact by Lemmata \ref{lem:c0} and \ref{lem:lip},
so we can extract $w_{\e_{j_k}} \rightarrow u$ in ${\mathcal K}_0(M,g)$.
Since $M_{\rho,g}$ is continuous for the uniform topology, we conclude that
$$
M_{\rho, g}[u]=\lim_k M_{\rho, g}[ w_{\e_{j_k}}]=
\lim_k e^{\e_{j_k} u_{\e_{j_k}}} \mu 
=c \mu.
$$

\medskip

\noindent {\it Uniqueness of $c$}.
 Suppose that $u, v$ are two $g$-convex functions satisfying  
$M_{\rho,g}[u]=c_1\mu$ and $M_{\rho,g}[v]=c_2\mu$. 
We now show that  $c_1=c_2$.
Assume by contradiction that $c_1>c_2$, so there is  $\delta>0$ such that $(1-\delta)^n c_1>c_2$. 
Set  $u_\delta = (1-\delta)u$ and   pick $x_{\delta} \in M$
such that
$$
\max_M (u_\delta-v) = u_\delta(x_\delta)-v(x_\delta)=:A. 
$$
Let $D$ be  a small neighborhood of $x_\delta$ such that $d_g( x_\delta, \partial D)\geq  d/3$ where $d= diam(D)$. 
Let $\phi$ be a potential of $g$ in $D$, i.e $g=\nabla d \phi$. 
Set $\hat u_\delta= \phi+ u_\delta$, $\hat u= \phi+ u$, $\hat v=\phi+v$ and  
\begin{eqnarray*}
\hat u_{\delta,\epsilon} 
&= &\hat u_\delta- \epsilon\|x-x_\delta\|^2 -(A- \epsilon d^2/10)  \\
&= & \phi + (1-\delta) u - \epsilon\|x-x_\delta\|^2 -(A- \epsilon d^2/10),
\end{eqnarray*}
where $\epsilon$ is so small so that $\nabla d(\delta \phi-\epsilon\|x-x_\delta\|^2) \geq (\delta/2)g$. 
Observe that
$$
\max_{\bar D} (\hat u_{\delta,\epsilon} -\hat v ) = \hat  u_{\delta,\epsilon}(x_\delta) -\hat v(x_\delta) =    \epsilon d^2/10 >0
$$
 and 
 $$
 \sup_{\partial D} (\hat  u_{\delta,\epsilon} -\hat v)  \leq A-\epsilon d^2/9 -(A-\epsilon d^2/10)<0. 
 $$
Thus the set  $B:= \{x\in \bar D|\, \hat u_{\delta,\epsilon} >\hat v \}$ is open, non empty  (as it contains $x_\delta$)
 and $  B\cap \partial D=\emptyset.$ We infer $\partial B = \{ x\in D|\, \hat u_{\delta,\epsilon} = \hat v \}$
 and Lemma \ref{lem:mass_com}  ensures that
$$
{\rm M}[\hat v] (B) \geq  {\rm M} [\hat u_{\delta,\epsilon}]  (B) \geq  (1-\delta)^n{\rm M} [ \hat u](B)+ (\delta/2)^n {\rm M}[\phi](B),
$$
since $ \nabla d( \hat u_{\delta,\epsilon}) \geq (1-\delta)^n\nabla d \hat u+ (\delta/2) g $.
   Therefore
\begin{eqnarray*}
c_2\mu(B) 
&\geq &  (1-\delta)^nc_1  \mu(B)+(\delta/2)^n  M_{\rho,g}[0](B) \\
&\geq & c_2\mu(B)+(\delta/2)^n  M_{\rho,g}[0](B).
\end{eqnarray*}
This implies that $M_{\rho,g}[0](B)=0$, a contradiction. Thus $c_1=c_2$ as claimed.
\end{proof}

\begin{remark}
The uniqueness of $u$ is more delicate. It is obtained in \cite{CV01} when
the solutions are ${\mathcal C}^2$-smooth by using a classical maximum principle; this requires  
  $\mu$ to be absolutely continuous with respect to some volume form, with H\"older density.
  
We make the observation that uniqueness holds in the most degenerate case when $\mu=\delta_p$
is a Dirac mass at a single point $p \in M$: if $M_{\rho,g}(u)=M_{\rho,g}(v)=c\delta_p$
with $u,v \in \K_0(M,g)$, then $M_{\rho,g}(\tilde{u})=e^{\tilde{u}} \delta_p$
and $M_{\rho,g}(\tilde{v})=e^{\tilde{v}} \delta_p$ with $\tilde{u}=u-u(p)+\log c$,
$\tilde{v}=v-v(p)+\log c$, so $\tilde{u}=\tilde{v}$ by uniqueness in Theorem \ref{thm:solutionMA},
which yields $u=v$.
\end{remark}

\subsection{Regularization of $g$-convex functions}

Let $(M,g)$ be a compact Hessian manifold and $\rho$ a $(-1)$-density.
Given a $g$-convex function $u$ on $M$, we set 
$\mu_u:=e^{-u} M_{\rho, g}[u]$ so that $u$ is the unique
$g$-convex solution of $M_{\rho, g}[u]=e^u \mu_u$.
Using convolutions we approximate $\mu_u$ by smooth volume forms
$\mu_{\e}=\mu \star \chi_{\e}+\e dV_M$
and invoke a result of Cheng-Yau \cite{CY82} to obtain a 
smooth strictly $g$-convex function $u_{\e}$ on $M$ such that
$$
M_{\rho, g}[u_{\e}]=e^{u_{\e}}\mu_{\e}.
$$

It follows from the stability  property (Theorem \ref{thm:solutionMA})
that $u_{\e}$ uniformly converges to $u$ as $\e \rightarrow 0^+$,
so any $g$-convex function $u$ is the uniform limit of smooth strictly
$g$-convex functions.
This provides an alternative proof of the global regularization of $g$-convex functions
(compare with Proposition \ref{prop:reg}).




\end{document}